\documentclass{IEEEtran}
\usepackage{amssymb,amsthm,amsmath}\usepackage{tikz}
\newcommand\R{\mathbb R}
\newcommand\eps{\varepsilon}
\newcommand\dom{\text{\upshape Dom}}

\newcommand\cost{\text{\upshape Cost}}
\newcommand\OPT{\text{\upshape OPT}}
\renewcommand\O{\R_{\geqslant0}}
\newcommand\nablaS{\nabla\mkern-2mu S}

\newtheorem{theorem}{Theorem}
\newtheorem{lemma}[theorem]{Lemma}
\newtheorem{claim}[theorem]{Claim}

\theoremstyle{definition}
\newtheorem{definition}[theorem]{Definition}
\newtheorem{problem}{Problem}

\theoremstyle{remark}
\newtheorem{remark}{Remark}

\let\oequation\equation
\let\oeqend\endequation
\makeatletter
\def\@thiseqno{NOTHING}
\def\@myendequation{\eqno \hbox{\normalfont \normalcolor \@thiseqno}$$\@ignoretrue}
\def\@eeequation[#1]{\let\@ooeqend\@myendequation
$$\edef\@thiseqno{#1}\protected@edef\@currentlabel{\csname p@equation\endcsname#1}}%
\renewenvironment{equation}{%
\let\@ooeqend\oeqend
\@ifnextchar[\@eeequation\oequation}
{\@ooeqend}
\makeatother

\title{An optimization problem for continuous submodular functions}
\author{Laszlo Csirmaz\thanks{%
Alfr\'ed R\'enyi Mathematical Institute, Budapest and UTIA, Prague}}
\IEEEspecialpapernotice{%
Dedicated to Professor Gheorghe Morosanu on the occasion of his
seventieth birthday}

\begin{document}
\maketitle

\begin{abstract}
Real continuous submodular functions, as a generalization of the
corresponding discrete notion to the continuous domain, gained considerable
attention recently. The analog notion for entropy functions requires additional
properties: a real function defined on the non-negative orthant of $\R^n$ is
entropy-like (EL) if it is submodular, takes zero at zero, non-decreasing,
and has the Diminishing Returns property.

Motivated by problems concerning the Shannon complexity of multipartite
secret sharing, a special case of the following general optimization problem
is considered: find the minimal cost of those EL functions which satisfy certain
constraints.

In our special case the cost of an EL function is the maximal value of the
$n$ partial derivatives at zero. Another possibility could be the supremum
of the function range. The constraints are specified by a smooth bounded
surface $S$ cutting off a downward closed subset. An EL function is feasible
if at the internal points of $S$ the left and right partial derivatives of
the function differ by at least one.

A general lower bound for the minimal cost is given in terms of the normals
of the surface $S$. The bound is tight when $S$ is linear. In the
two-dimensional case the same bound is tight for convex or concave $S$. It
is shown that the optimal EL function is not necessarily unique. The paper
concludes with several open problems.
\end{abstract}

\section{Introduction}\label{sec:intro}

Continuous submodularity is a generalization of the discrete notion of
submodularity to the continuous domain. It has gained considerable attention
recently \cite{bach-submodular,bian2020continuous} as efficient convex
optimization methods can be extended to find the minimal and maximal value
of special multivariable continuous submodular functions over a compact and
convex domain. Such optimization algorithms have important applications in
many areas of computer science and applied mathematics such as training deep
neural networks \cite{botton-et-al}, design of online experiments
\cite{chen-2018}, or budget allocation \cite{staib-2017}.
For more information see \cite{bach-learning-submodular}.

Interestingly, the same class of continuous submodular functions arises when
the continuous version of multipartite secret sharing schemes is considered.
In classical secret sharing \cite{beimel-survey} each participant receives a
piece of information -- their \emph{shares} -- such that a qualified subset
of participants can recover the secret from the shares they received, while
unqualified subsets -- based on their shares only -- should have no
information on the secret's value at all. In the multipartite case
\cite{padro:idealmulti,padro:optimization} participants are in $n$ disjoint
groups, and members in the same group have equal roles. In particular, a
qualified subset is described uniquely by the $n$ numbers telling how many
members this subset has from each group. The main question in secret sharing
is the efficiency -- also called complexity -- of the scheme, which is
typically defined as the worst-case ratio of the size of any of the shares
(measured by their Shannon entropy) and the size of the secret. Keeping
track of the total entropy of different subsets of shares, traditional
entropy inequalities imply a lower bound on the complexity
\cite{padro:optimization,metcalf} known as the \emph{Shannon-bound}. No
general method is known which would effectively determine, or even estimate,
the Shannon bound for an arbitrary collection of qualified subsets, and
numerical computation is intractable even for moderately sized problems.
Investigating the same question in the continuous domain allows applying
analytical tools, and results achieved this way might shed light on the
discrete case. This paper, based partly on the last section of \cite{CMP},
is an attempt to initiate such a line of research.

No notion from secret sharing or from information theory will be used later
as they only serve as motivation for the definitions. The family of real
functions corresponding to the (normalized) multipartite entropy will be
called \emph{entropy-like} functions and abbreviated as EL. This function
family is defined in Section \ref{sec:Definitions}; actually it is the
family of pointed, increasing, submodular functions with the ``Diminishing
Returns'' property, see \cite{bian2020continuous}.

The optimization problem corresponding to finding an optimal multipartite
secret sharing scheme is discussed in Section \ref{sec:problem}. It differs
from the well-studied optimization problem for submodular functions
\cite{bach-submodular,bian2020continuous}, where some member of the
continuous submodular function family is given, and the task is to find its
maximal (minimal) value over a compact, convex set. In our case the
optimization problem asks to find an EL function with the smallest cost
satisfying certain constraints. Two cost functions are considered. The first
one corresponds to the discrete worst case complexity discussed above, and
it is the maximal partial derivative of the EL function at the origin. The
second possibility is the supremum of the function range; it corresponds to
another frequently investigated complexity measure in the discrete case: the
total randomness used by the scheme. In Section \ref{sec:problem} a general
lower bound for the worst case complexity is given as Theorem
\ref{thm:main}. This bound is tight when the constraints are
specified by some linear surface.

Section \ref{sec:bipartite} presents results for the bipartite,
two-di\-men\-sion\-al case. General constructions show that the lower bound of
Theorem \ref{thm:main} is also tight for strictly convex or strictly concave
constraint curves. An alternate construction shows that the optimal EL function
is not necessarily unique. Finally, Section \ref{sec:conclusion} concludes
the paper with a list of open problems.

\section{Submodular and entropy-like functions}\label{sec:Definitions}

A real function $f$ defined on subsets of a set is \emph{submodular} if $f(A)+f(B)\ge
f(A\cap B)+f(A\cup B)$ for arbitrary subsets $A$ and $B$, see \cite{bach-submodular} and references therein.
The same notion extended to an arbitrary lattice requires
$$
   f(A)+f(B)\ge f(A\land B)+f(A\lor B)
$$
for any two lattice members $A$ and $B$. 
In particular,
the $n$-variable real function $f$ is \emph{submodular} if it is submodular in
the lattice determined by the partial order on $\R^n$ defined by $x\le y$ if
and only if $x_i\le y_i$ for all coordinates $1\le i\le n$. In this case
$x\land y=\min(x,y)$ and $x\lor y=\max(x,y)$ where minimization
(maximization) is taken coordinatewise, and the submodularity condition
rewrites to
$$
   f(x)+f(y)\ge f(\min(x,y))+f(\max(x,y)).
$$
Entropy-like real functions, also called EL functions, share
additional properties with discrete Shannon entropy functions
\cite{yeung-book}, and are defined as follows.
\begin{definition}\label{def:EL-function}
The $n$-variable real function $f$ is \emph{entropy-like}, or EL function
for short, if it satisfies properties (a) -- (e) below.
\begin{itemize}
\item[(a)] $f$ is defined on the non-negative orthant $\O^n=\{x\in\R^n: x\ge 0\}$.
\item[(b)] $f$ is submodular.
\item[(c)] $f(0)=0$ ($f$ is pointed).
\item[(d)] $f$ is non-decreasing: if $0\le x\le y$ then $f(x)\le f(y)$.
\item[(e)] $f$ has the ``Diminishing Returns'' property \cite{bian2020continuous}. 
It means that for two points 
$0\le x\le y$ differing only in their $i$-th coordinate, increasing that coordinate 
at $x$ and also 
at $y$ by the same amount $\eps$, the gain at $y$ is never bigger
than the gain at $x$. Formally, if
$e_i$ is the $i$-th unit vector and $y=x+\lambda e_i$ for some $\lambda>0$, then 
for every $\eps>0$,
\begin{equation}\label{eq:diminishing}
    f(x+\eps e_i)-f(x) \ge f(y+\eps e_i)-f(y).
\end{equation}
\end{itemize}
\end{definition}
The ``Diminishing Returns'' property models the natural expectation that
adding one more unit of some resource contributes more in the case when one has less available 
amount of that resource.

The left and right partial derivatives of the $n$-variable function $f$ at
$x\in\R^n$ are denoted by $f^-_i(x)$ and $f^+_i(x)$, respectively, and their
definition goes as
\begin{align*}
f^-_i(x)&=\lim_{\eps\to+0}\frac{f(x)-f(x-\eps e_i)}{\eps} \\[-5pt]
\intertext{and}\\[-30pt]
 f^+_i(x)&=\lim_{\eps\to+0}\frac{f(x+\eps e_i)-f(x)}{\eps}
\end{align*}
assuming that the corresponding limits exist. Here $e_i$ is the $i$-th unit vector.

The following claim summarizes some basic properties of
EL functions.
\begin{claim}\label{claim:EL-basic}
Let $f$ be an $n$-variable EL function.
\begin{itemize}
\item[\upshape(a)] $f$ is continuous.
\item[\upshape(b)] $f$ is concave along any positive direction: if $0\le x\le y$
and $0\le\lambda\le1$ then
$$
   \lambda f(x)+(1-\lambda)f(y) \le f(\lambda x + (1-\lambda)y) .
$$
\item[\upshape(c)] The Diminishing Returns property
{\upshape(\ref{eq:diminishing})} holds for arbitrary pair of points $0\le x\le y$.
\item[\upshape(d)] $f$ has both left and right partial derivatives at every point
of its domain.
\item[\upshape(e)] The partial derivatives are non-negative and non-increasing
along any positive direction.
\end{itemize}%
\end{claim}
\begin{proof}
(a) It is enough to show that $f$ is continuous along every coordinate. By
property (d) it is monotone increasing. The left limit
$\lim_{\eps\to+0} f(x-\eps e_i)$ cannot be strictly smaller than the right
limit
$\lim_{\eps\to+0} f(x+\eps e_i)$ as this would contradict the Diminishing
Returns property.

\def\oi{\mathbin{\ooalign{$\hidewidth|\hidewidth$\cr$\circ$}}}
(b) Continuity and the Diminishing Returns property ensures that $f$ is
concave along each coordinate. It means that statement (b) is true when 
points $x$ and $y$ share $n-1$ coordinates. Suppose we have two points
sharing $i$ coordinates, and the claim has been established for point pairs 
sharing $i+1$ or more
coordinates. Denote these points by $(c,x,a)$ and $(d,y,a)$ where $a$
stands for the joint $i$ coordinates, $x$ and $y$ are real numbers, and $c$ and $d$
are the remaining tuples. The linear
combination $\lambda(c,x,a)+(1-\lambda)(d,y,a)$ is shortened to
$(c\oi d,x\oi y,a)$. Using $(c,x,a)\le (d,y,a)$ and the induction hypothesis
for $n-1$ (first line) and for $i+1$ (next two lines) we have
\begin{align*}
   \lambda f(c\oi d,x,a)+(1-\lambda)f(c\oi d,y,a)&\le f(c\oi d,x\oi y,a), \\
   \lambda f(c,x,a)+(1-\lambda)f(d,x,a)&\le f(c\oi d,x,a), \\
   \lambda f(c,y,a)+(1-\lambda)f(d,y,a)&\le f(c\oi d,y,a).
\end{align*}
From here the required inequality
$$
   \lambda f(c,x,a)+(1-\lambda)f(d,y,a)\le f(c\oi d,x\oi y,a)
$$
follows as the submodularity for the points $(c,y,a)$ and ($d,x,a)$ gives
$$
  f(c,y,a)+f(d,x,a)\ge f(c,x,a)+f(d,y,a).
$$

(c) Similarly to (b) by induction on how many coordinates $x$ and $y$
have in common. Observe that if $x$ and $y$ do not differ at their 
$i$-th coordinate then (\ref{eq:diminishing}) is equivalent to
submodularity.

(d) This is immediate as $f$ is continuous and non-decreasing.

(e) Non-negativity is clear. Monotonicity: if $x\le y$ then, for example,
\begin{align*}
   f^+_i(x) &=\lim_{\eps\to+0}\frac{f(x+\eps e_i)-f(x)}{\eps} \\
            &\ge\lim_{\eps\to+0}\frac{f(y+\eps e_i)-f(y)}{\eps}
             = f^+_i(y),
\end{align*}
where the inequality follows from (c). Other cases are similar.
\end{proof}

The next lemma follows easily from the fact that along each coordinate $f$
is increasing and concave, and is given without proof.

\begin{lemma}\label{lemma:limiting}
If $\eps\to+0$, then $f^+_i(x+\eps e_i)\to f^+_i(x)$, and
$f^+_i(x-\eps e_i)\to f^-_i(x)$.
\qed
\end{lemma}

\begin{remark}\label{remark:1}
The family of EL functions is closed for non-negative linear combination and
truncation: 
if $f_1$, $f_2$ are EL, then so is $\lambda_1f_1+\lambda_2f_2$ for
$\lambda_1$, $\lambda_2\ge 0$; if $f$ is EL and $M\ge 0$ then 
$\min(f,M)$ is EL. Consequently
$$
   f(x)=\min\big(\, {\textstyle\sum c_ix_i},M\big)
$$
is EL for positive $c_i$ and $M$. Similarly, if $f$ is EL and $a\ge 0$, then
$g(x)=f(\min(x,a))$ is EL again.
Further examples of EL functions
will be given in Section \ref{sec:bipartite}.
\end{remark}

\begin{remark}\label{remark:2}
If the sequence $f_k$ of EL functions converge pointwise, then the limit $f$ is
also an EL function, moreover
$$
   f^+_i(x) \le \liminf_k\,(f_k)^+_i(x) \le \limsup_k\,(f_k)^-_i(x)\le f^-_i(x).
$$
\end{remark}

\section{The optimization problem}\label{sec:problem}

According to the intuition discussed in Section \ref{sec:intro} the value of
$n$-variable EL function $f$ at $x\in\O^n$ can be considered as the value of the
(scaled) entropy of the set of shares assigned to a subset of participants which
has members from the $i$-th group proportional to the $i$-th coordinate of
$x$. The right derivative $f^+_i(x)$ can be interpreted as the (scaled)
entropy increase if one more member from the $i$-th group joins this subset,
and $f^-_i(x)$ as the entropy decrease when one member from the $i$-th group
leaves the subset (defined only if $x_i>0$). Consequently the share size of 
a single participant from
group $i$ can be identified to $f^+_i(0)$, the $i$-th right partial derivative of
$f$ at zero. Accordingly, the cost function corresponding to the
maximal share size is
$$
    \cost(f) = \max \{ f^+_1(0), f^+_2(0),\dots,f^+_n(0)\,\}.
$$
While this cost function will be considered in this paper, there are other
possibilities. In the discrete cases the total entropy (the amount of 
randomness needed to generate the whole scheme) is used frequently, this
would correspond to the cost function $\sup\{ f(x): x \in\dom(f) \}$.

In secret sharing the shares of a qualified subset determine the secret,
while the same secret is (statistically) independent of the shares of an
unqualified subset. We call the point $x\in\O^n$ \emph{qualified} if the
corresponding subset is qualified. When decreasing an unqualified subset it
remains unqualified, thus the set of unqualified points are downward closed: if $x$
is unqualified and $0\le y\le x$ then $y$ is unqualified as well. Suppose
the unqualified and qualified points are separated by the smooth
$(n-1)$-dimensional surface $S$. Downward closedness means that the normal
vectors of $S$ pointing outwards (towards qualified points) have
non-negative coordinates. This surface $S$ specifies the secret sharing
problem, namely which subsets of the participants are qualified
and which are not, and thus the optimization problem as well.
The definition below requires slightly stronger properties from such a
separating surface excluding certain problematic cases.
\begin{definition}\label{def:s-surface}
An \emph{s-surface} (secret sharing surface) is a smooth $(n-1)$-dimensional
surface $S$ in the non-negative orthant $\O^n$ satisfying the following properties:
\begin{itemize}
\item[(a)] $S$ avoids $0$,
\item[(b)] $S$ is compact, and
\item[(c)] for every $x\in S$ the normal vector $\nablaS(x)$ pointing outwards has
strictly positive coordinates.
\end{itemize}%
\end{definition}
Consider the subset of participants which corresponds to the point $x\in S$
of the s-surface $S$.
If any member from the $i$-th group leaves this subset, then the subset
becomes unqualified -- and then the secret must be independent of the joint
collection of the associated shares. If any new member from the $i$-th group
joins that subset, it becomes qualified -- meaning that the new share
collection determines the secret. Thus the difference between the before and
after entropy changes, namely $f^-_i(x)-f^+_i(x)$, must cover the entropy
of the secret. The entropy of the secret can be taken to be 1 as this
changes all values up to a scaling factor only. The following definition
summarizes this discussion.
\begin{definition}\label{def:feasible}
The EL function $f$ is \emph{feasible} for $S$, or $S$-feasible,
if for every positive $x\in S$ (that is, $x_i>0$ for all $1\le i\le n$),
\begin{equation}\label{eq:condition}
   f^-_i(x)-f^+_i(x)\ge 1 ~~~ (1\le i\le n).
\end{equation}
(Positivity of $x$ is ensures the existence of $f^-_i(x)$.)
\end{definition}
Optimization problems considered in this paper are of this form: 
given the s-surface $S$, find the minimal cost of the $S$-feasible functions.
\begin{definition}\label{def:opt-problem}
For a given s-surface $S\subseteq \O^n$ $\OPT(S)$ is the optimization
problem
$$
\left\{\mkern -5mu
\begin{array}{ll}
\text{minimize:} & \cost(f) \\[3pt]
\text{subject to:} & \text{$f$ is an $S$-feasible EL function}.
\end{array}\right.
$$
By an abuse of notation, both the problem and its solution -- the infimum of
the costs of $S$-feasible functions -- will be denoted
by $\OPT(S)$.
\end{definition}

As an example let us consider the case when $S$ is the
intersection of the hyperplane
$$
   c_1x_1+c_2x_2+\cdots+c_nx_n=M
$$
and the non-negative orthant,
here $c_i$ and $M$ are positive constants. Observe that the normal at every
$x\in S$ is $\nablaS(x)\allowbreak=\allowbreak(c_1,\dots,c_n)$. Feasible EL
functions will be searched among the one-parameter family
$$
   f(y) = k\cdot\min\big\{{\textstyle\sum}c_iy_i, M\big\}
$$
with positive $k$. All of them are EL functions by Remark \ref{remark:1}.
Pick the positive point $x\in S$ 
and consider $f(x+\eps e_i)$ as a function of $\eps$. It has the constant
value $k\cdot M$ for $\eps\ge 0$, and it is linear with slope $k\cdot c_i$
for $\eps\le 0$. Consequently
$$
   f^-_i(x)-f^+_i=k\cdot c_i,
$$
which is $\ge 1$ if $k\ge1/\min\{c_i\}$. At zero the partial derivatives of $f$
are $k\cdot c_i$, therefore $\cost(f)=k\cdot\max\{c_i\}$. The 
$k=1/\min\{c_i\}$ choice gives an $S$-feasible EL function with cost
$\max\{c_i\}/\min\{c_i\}$, thus
$$
   \OPT(S)\le \frac{\max\{c_i\}}{\min\{c_i\}}.
$$
According to Theorem \ref{thm:main} below the optimal value is actually equal to this
amount, as in this case $\nablaS_i(x)=c_i$ for every $x\in S$.

\begin{theorem}\label{thm:main}
For every s-surface $S$, inner point $x\in S$ and $1\le i,j\le n$ the following
inequality holds:
$$
   \OPT(S) \ge \frac{\nablaS_j(x)}{\strut\nablaS_i(x)}.
$$
\end{theorem}
\begin{proof}
By assumption $S$ behaves linearly on a small neighborhood of $x$, thus
for every small enough positive $w$
there is a unique positive $h$ such that $y=x-w e_i+h e_j\in S$, and
$$
    \lim_{w\to+0} \frac h w = \frac{\nablaS_j(x)}{\strut\nablaS_i(x)}.
$$
Let $f$ be any $S$-feasible EL function, $u=\min(x,y)=x-w e_i$ and 
$v=\max(x,y)=x+h e_j$. The following inequalities follow from the facts that
$f$ is monotone and concave along each coordinate by Claim
\ref{claim:EL-basic}:
\begin{align*}
w\cdot f^+_i(u) &\ge f(x)-f(u), \\
h\cdot f^+_j(x) &\ge f(v)-f(x), \\
f(y)-f(u) &\ge h\cdot f^-_j(y), \\
f(v)-f(y) &\ge 0.
\end{align*}
Their sum proves the first inequality in the sequence
\begin{align*}
   w\cdot f^+_i(u) &\ge h\big(f^-_j(y)-f^+_j(x)\big) \\
          &\ge h\big(1+f^+_j(y)-f^+_j(x)\big) \\
          &\ge h\big(1+f^+_j(v)-f^+_j(x)\big).
\end{align*}
The second inequality follows from $y\in S$ and that $f$ is an
$S$-feasible function. The third one uses the monotonicity of the
derivatives from Claim \ref{claim:EL-basic} (e).
Letting $w\to+0$, $f^+_i(u)\to f^-_i(x)$ and $f^+_j(v)\to f^+_j(x)$ by Lemma
\ref{lemma:limiting}, thus
$$
    f^-_i(x) \ge \frac{\nablaS_j(x)}{\strut\nablaS_i(x)}.
$$
From here the theorem follows as $\cost(f)\ge f^+_i(0) \ge f^-_i(x)$
by the monotonicity of the derivatives.
\end{proof}

\begin{theorem}\label{thm:inf-taken}
Suppose $\OPT(S)<+\infty$ for an s-surface $S$. The optimal value is
taken by some $S$-feasible function $f$, that is, $\cost(f)=\OPT(S)$.
\end{theorem}
\begin{proof}
Let $\OPT(S)<M$, and choose the sequence of $S$-fea\-si\-ble functions $f_k$ such
that $\cost(f_k)<M$ and $\lim_k\cost(f_k)=\OPT(S)$. Also pick a point $a\in\O^n$
such that $S$ is contained completely in the box $B=\{x\in\O^n: x\le a\}$.
The functions $g_k(x)=f(\min(x,a))$ are EL by Remark \ref{remark:1}, and
$\cost(g_k)=\cost(f_k)$. Each $g_k$ is clearly $S$-feasible and is bounded by
$M\cdot(a_1+\cdots+a_n)$. The sequence $\{g_k\}$ is uniformly equicontinuous 
as all partial derivatives are bounded by $M$, thus the 
Arzel\`a--Ascoli theorem \cite{rudin-1976} guarantees a subsequence which
converges uniformly on $B$ -- and then converges everywhere. Denote this
subsequence also by $\{g_k\}$, and let the pointwise limit be $g$. By Remark \ref{remark:2}
$g$ is an EL function and $\cost(g)\le\liminf_k\cost(g_k)=\OPT(S)$. Also,
each $g_k$ is $S$-feasible, that is, at the points of $S$ the difference 
between the left and right derivatives is at least 1:
$$
  (g_k)^-_i(x) - (g_k)^+_i(x)\ge 1, ~~~~ x\in S.
$$
By Remark \ref{remark:2} the same is true for the limit function $g$. Thus there is
an $S$-feasible function $g$ with $\cost(g)\le\OPT(S)$, which proves the
theorem.
\end{proof}

\section{Two-dimensional cases}\label{sec:bipartite}

We have seen that the bound provided by Theorem \ref{thm:main} is sharp when
$S$ is linear. We show that, at least in the two-dimensional case, it is
also sharp when $S$ is strictly convex or strictly concave by constructing
matching $S$-feasible EL functions.

In two dimensions $S$ is a strictly decreasing continuous curve. Write $S$
as $\{(x,\alpha(x)): 0\le x\le a\}$, and also as $\{(\beta(y),y):0\le y\le
b\}$, see Figure \ref{fig:S}.
\begin{figure}[thb]
\noindent\hbox to\linewidth{\space
\begin{tikzpicture}[scale=3]
\fill[color=gray!10!white]
(0,0)--(0,0.391)--(0.461,0.391)--(0.461,0)--cycle;
\fill[color=gray!8!white]
(0,0.391)--(0.461,0.391)--(0.461,0.931)--(0,0.931)--cycle;
\fill[color=gray!8!white]
(0.461,0.391)--(0.461,0)--(1.09,0)--(1.09,0.391)--cycle;
\draw[thick]
(0.000,0.929) -- (0.019,0.875) -- 
(0.052,0.821) -- (0.094,0.768) -- 
(0.142,0.714) -- (0.194,0.660) -- 
(0.246,0.606) -- (0.300,0.552) -- 
(0.354,0.498) -- (0.407,0.444) --
(0.461,0.391) -- (0.515,0.339) -- 
(0.569,0.291) -- (0.623,0.245) -- 
(0.677,0.202) -- (0.731,0.162) -- 
(0.785,0.124) -- (0.838,0.089) -- 
(0.892,0.057) -- (0.946,0.027) -- 
(1.000,0.000);
\draw[->] (0,0) -- (1.1,0);
\draw[->] (0,0) -- (0,0.96);
\draw (0.461,0.391) node {$\bullet$} (0.461,0.391) node [above right] {$T$};
\draw (-0.05,0.391) -- (0.461,0.391)-- (0.461,-0.05);
\draw (-0.05,0.391) node [left] {$t_y$} (-0.05,-0.05) node {$0$}
      (0.461,-0.05) node [right] {$t_x$};
\draw (1,0,0) node [below] {$a$}
      (0,0.93) node [left] {$b$};
\draw  (0.194,0.660) node [right]{$\alpha(x)$};
\draw (0.731,0.162) node [right] {$\beta(y)$};
\end{tikzpicture}\hss}\par\kern-6pt%
\caption{The curve $S$}\label{fig:S}%
\end{figure}
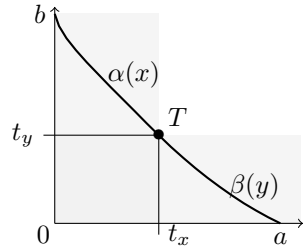

If $S$ is either convex or concave, then $\nablaS_i(x) / \nablaS_j(x)$ is
increasing or decreasing along the curve, thus attains its maximal value at
one of the endpoints.

First assume that $S$ is strictly convex. In this case both 
$\alpha$ and $\beta$ are convex functions. Let
$T=(t_x,t_y)$ be the point on $S$ where the normal is $(1,1)$. On the $[0,t_x]$
interval the derivative $\alpha'(x)$ is $\le -1$, and, similarly, $\beta'(y)\le -1$ on
$[0,t_y]$. The function $f$ depicted on Figure \ref{fig:convex1} is
defined as follows.
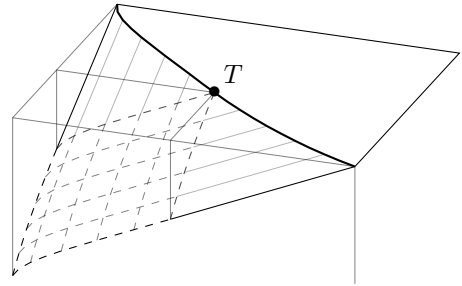
\begin{figure}[b]%
\noindent\hbox to\linewidth{\quad
\begin{tikzpicture}[scale=0.65,x={(7.0cm,-1cm)},y={(2.3cm,2.5cm)},z={(0cm,3cm)}]
\draw[very thin,opacity=0.3]
(0.569,0.291,0.0) --  (0.461,0.291,-0.108)
(0.677,0.202,0.0) -- (0.461,0.202,-0.215)
(0.785,0.124,0.0) -- (0.461,0.124,-0.323)
(0.892,0.057,0.0) -- (0.461,0.057,-0.431);
\draw[very thin,opacity=0.3]
(0.052,0.821, 0.0) -- (0.052,0.391,-0.431)
(0.142,0.714, 0.0) -- (0.142,0.391,-0.323)
(0.246,0.606, 0.0) -- (0.246,0.391,-0.215)
(0.354,0.498, 0.0) -- (0.354,0.391,-0.108);

\foreach \y/\z in {0.291/-0.108,0.202/-0.215,0.124/-0.323,0.057/-0.431}{
\draw[very thin,opacity=0.5,dashed]
(0.000,\y,\z-0.539) -- (0.019,\y,\z-0.485) -- 
(0.052,\y,\z-0.431) -- (0.094,\y,\z-0.377) -- 
(0.142,\y,\z-0.323) -- (0.194,\y,\z-0.269) -- 
(0.246,\y,\z-0.215) -- (0.300,\y,\z-0.162) -- 
(0.354,\y,\z-0.108) -- (0.407,\y,\z-0.054) --
(0.461,\y,\z+0.000);
}

\foreach \x\z in {0.354/-0.108,0.246/-0.215,0.142/-0.323,0.052/-0.431}{
\draw[very thin,opacity=0.5, dashed]
(\x,0.391,\z+0.000) -- (\x,0.339,\z-0.054) -- 
(\x,0.291,\z-0.108) -- (\x,0.245,\z-0.162) -- 
(\x,0.202,\z-0.215) -- (\x,0.162,\z-0.269) -- 
(\x,0.124,\z-0.323) -- (\x,0.089,\z-0.377) -- 
(\x,0.057,\z-0.431) -- (\x,0.027,\z-0.485) -- 
(\x,0.000,\z-0.539);
}

\draw[very thin,opacity=0.5](0,0.929,0)--(0,0,0)--(1,0,0)--(1,0,-0.8);
\draw (1,0,0)--(1,0.929,0)--(0,0.929,0);
\draw (0.000,0.391,-0.539)--(0,0.929,0) (1,0,0)--(0.461,0.0,-0.539);
\draw[very thin, opacity=0.5](0,0,0)--(0,0,-1.077);
\draw[very thin,opacity=0.5] (0.461,0.391, 0.0)--(0.461,0,0)-- (0.461,0.0,-0.539);
\draw[very thin,opacity=0.5] (0.461,0.391, 0.0)--(0,0.391,0) -- (0.000,0.391,-0.539);

\draw (0.461,0.391, 0.0) node {$\bullet$};
\draw (0.461,0.391, 0.0) node[above right] {$T$};
\draw[thick]
(0.000,0.929, 0.0) -- (0.019,0.875, 0.0) -- 
(0.052,0.821, 0.0) -- (0.094,0.768, 0.0) -- 
(0.142,0.714, 0.0) -- (0.194,0.660, 0.0) -- 
(0.246,0.606, 0.0) -- (0.300,0.552, 0.0) -- 
(0.354,0.498, 0.0) -- (0.407,0.444, 0.0) -- (0.461,0.391, 0.0);
\draw[dashed,opacity=0.7]
(0.000,0.391,-0.539) -- (0.019,0.391,-0.485) -- 
(0.052,0.391,-0.431) -- (0.094,0.391,-0.377) -- 
(0.142,0.391,-0.323) -- (0.194,0.391,-0.269) -- 
(0.246,0.391,-0.215) -- (0.300,0.391,-0.162) -- 
(0.354,0.391,-0.108) -- (0.407,0.391,-0.054) -- (0.461,0.391,0.000);
\draw[dashed]
(0.000,0.0,-1.077) -- (0.019,0.0,-1.023) -- (0.052,0.0,-0.970) -- (0.094,0.0,-0.916) -- (0.142,0.0,-0.862) -- (0.194,0.0,-0.808) -- (0.246,0.0,-0.754) -- (0.300,0.0,-0.700) -- (0.354,0.0,-0.646) -- (0.407,0.0,-0.593) -- (0.461,0.0,-0.539);
\draw[thick]
(0.461,0.391,0.0) -- (0.515,0.339,0.0) -- (0.569,0.291,0.0) -- (0.623,0.245,0.0) -- (0.677,0.202,0.0) -- (0.731,0.162,0.0) -- (0.785,0.124,0.0) -- (0.838,0.089,0.0) -- (0.892,0.057,0.0) -- (0.946,0.027,0.0) -- (1.000,0.000,0.0);
\draw[dashed,opacity=0.7]
(0.461,0.391,0.000) -- (0.461,0.339,-0.054) -- (0.461,0.291,-0.108) -- (0.461,0.245,-0.162) -- (0.461,0.202,-0.215) -- (0.461,0.162,-0.269) -- (0.461,0.124,-0.323) -- (0.461,0.089,-0.377) -- (0.461,0.057,-0.431) -- (0.461,0.027,-0.485) -- (0.461,0.000,-0.539);
\draw[dashed]
(0.0, 0.391,-0.539) -- (0.0, 0.339,-0.593) -- (0.0, 0.291,-0.646) -- (0.0, 0.245,-0.700) -- (0.0, 0.202,-0.754) -- (0.0, 0.162,-0.808) -- (0.0, 0.124,-0.862) -- (0.0, 0.089,-0.916) -- (0.0, 0.057,-0.970) -- (0.0, 0.027,-1.023) -- (0.0, 0.000,-1.077);
\end{tikzpicture}\hss}\par\kern -6pt%
\caption{Convex case}\label{fig:convex1}%
\end{figure}
If both $x\ge t_x$ and $y\ge t_y$ then $f(x,y)=C$, otherwise
$$
f(x,y)=\left\{\mkern-4mu\begin{array}{ll}
C+\min\{x-\beta(y),0\} & \text{if $x\ge t_x$,}\\
C+\min\{y-\alpha(x),0\} & \text{if $y\ge t_y$,}\\
a-\alpha(x)+b-\beta(y) & \text{otherwise,}
\end{array}\right.
$$
where $C=a-t_x+b-t_y$. Clearly $f$ has a flat plateau of height $C$
beyond the curve $S$. It is a routine
to check that $f$ is an EL function; one has to use that $-\alpha(x)$ and
$-\beta(y)$ are concave functions and have derivative $1$ at $x=t_x$ and
$y=t_y$, respectively. The left and right partial derivatives of $f$ at
$(x,y)\in S$ are $(1,0)$ and $(-\beta'(y),0)$ when $x\ge t_x$, and
$(-\alpha'(x),0)$ and $(1,0)$ when $y\ge t_y$. In all cases the values in
the pair differ by
at least one, thus $f$ is a feasible $S$-function. The partial
derivatives of $f$ at zero are $-\alpha'(0)$ and $-\beta'(0)$, thus
$$
    \cost(f)=\max\{ -\alpha'(0), -\beta'(0) \}
$$
matching the lower bound of Theorem \ref{thm:main}.

In the case when no point on $S$ has normal $(1,1)$ the simpler construction
using only the first (or second) line in the definition of the function $f$
works.

\smallskip

A different construction is illustrated on Figure \ref{fig:convex2}
which also meets the lower bound of Theorem \ref{thm:main}. It also shows that the
optimal EL function, if exists, is not necessarily unique.
Using the same notation as above, 
\begin{figure}[t]
\noindent\hbox to\linewidth{\quad
\begin{tikzpicture}[scale=0.65,x={(7.0cm,-1cm)},y={(2.3cm,2.5cm)},z={(0cm,3cm)}]
\foreach \x\y in {0.569/0.291,0.677/0.202,0.785/0.124,0.892/0.057}{
\draw[very thin,opacity=0.3] (\x,\y,0)--(\x,0,\y*-1.263);
}
\foreach \x\y in {0.052/0.821,0.142/0.714,0.246/0.606,0.354/0.498}{
\draw[very thin,opacity=0.3] (\x,\y,0)--(0,\y,\x*-1.262);
}

\foreach \x in {2,4,6,8}{
\draw[very thin,opacity=0.3]
 (0,{(10-\x)*0.0391},{-(10-\x)*0.0582-\x*0.1077}) -- (0.461,{(10-\x)*0.0391},-\x*0.0494)
 (\x*0.0461,0.391,{(10-\x)*-0.0582}) -- (\x*0.0461,0,{(10-\x)*-0.1077+\x*-0.0494});
}

\draw[very thin, opacity=0.5](0,0.929,0)--(0,0,0)--(1,0,0)--(1,0,-0.8);
\draw(0,0.929,0)--(1,0.929,0)--(1,0,0);
\draw[very thin, opacity=0.5](0,0,0)--(0,0,-1.077);
\draw[very thin,opacity=0.5] (0.461,0.391, 0.0)--(0.461,0,0)-- (0.461,0.0,-0.494);
\draw[very thin,opacity=0.5] (0.461,0.391, 0.0)--(0,0.391,0) -- (0.000,0.391,-0.582);
\draw[opacity=0.7] (0.000,0.391,-0.582)--(0.461,0.391,0.0)--(0.461,0.0,-0.494);
\draw (0.461,0.0,-0.494)--(0,0,-1.077)--(0.000,0.391,-0.582);
\draw[thick]
(0.000,0.929, 0.0) -- (0.019,0.875, 0.0) -- 
(0.052,0.821, 0.0) -- (0.094,0.768, 0.0) -- 
(0.142,0.714, 0.0) -- (0.194,0.660, 0.0) -- 
(0.246,0.606, 0.0) -- (0.300,0.552, 0.0) -- 
(0.354,0.498, 0.0) -- (0.407,0.444, 0.0) -- (0.461,0.391, 0.0);
\draw[thick]
(0.461,0.391,0.0) -- (0.515,0.339,0.0) -- 
(0.569,0.291,0.0) -- (0.623,0.245,0.0) -- 
(0.677,0.202,0.0) -- (0.731,0.162,0.0) -- 
(0.785,0.124,0.0) -- (0.838,0.089,0.0) -- 
(0.892,0.057,0.0) -- (0.946,0.027,0.0) -- (1.000,0.000,0.0);
\draw ((0.461,0.391, 0.0) node {$\bullet$};
\draw (0.461,0.391, 0.0) node[above right] {$T$};
\foreach \x in {-1.263}{
\draw[dashed]
(0.461,0,\x*0.391) -- (0.515,0,\x*0.339) -- 
(0.569,0,\x*0.291) -- (0.623,0,\x*0.245) -- 
(0.677,0,\x*0.202) -- (0.731,0,\x*0.162) -- 
(0.785,0,\x*0.124) -- (0.838,0,\x*0.089) -- 
(0.892,0,\x*0.057) -- (0.946,0,\x*0.027) -- (1.000,0.000,0.0);
}
\foreach \x in {-1.262}{
\draw[dashed]
(0,0.929,  0.0) --    (0,0.875,\x*0.019) -- 
(0,0.821,\x*0.052) -- (0,0.768,\x*0.094) -- 
(0,0.714,\x*0.142) -- (0,0.660,\x*0.194) -- 
(0,0.606,\x*0.246) -- (0,0.552,\x*0.300) -- 
(0,0.498,\x*0.354) -- (0,0.444,\x*0.407) -- (0,0.391,\x*0.461);
}
\end{tikzpicture}%
\hss}\par\kern-6pt%
\caption{Alternate construction for the convex case}\label{fig:convex2}%
\end{figure}
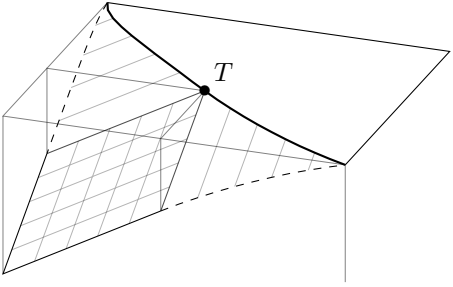
the function $f$ is defined analogously: $f(x,y)=C$ if $x\ge t_x$ and $y\ge
t_y$, otherwise
$$
  f(x,y)=\left\{\mkern-4mu\begin{array}{ll}
  C+\min\{y-\alpha(x),0\} & \text{if $x\ge t_x$,}\\
  C+\min\{x-\beta(y),0\} &  \text{if $y\ge t_y$,}\\
  x+y & \text{otherwise,}
\end{array}\right.
$$
where $C=t_x+t_y$. This is again an EL function, its cost is clearly $1$.
The difference between the left and right partial derivatives at points of
$S$ are $-\alpha'(x)$ and $1$ when $x\ge t_x$, and $1$ and $-\beta'(y)$ when
$y\ge t_y$, thus the difference is at least $k=\min\{-\alpha'(a), -\beta'(b)\}$.
Consequently the EL function $k^{-1}f(x,y)$ is feasible for $S$, and its
cost, $1/k$, matches the lower bound in Theorem \ref{thm:main}.


\smallskip

The third construction, depicted on Figure \ref{fig:concave}, works for any
strictly
concave curve $S$. In this case the plateau is not flat any more. Using the
same notations as before, the decreasing functions $\alpha(x)$ and $\beta(y)$
are strictly concave, and $T=(t_x,t_y)$ is the curve point with normal
$(1,1)$.
\begin{figure}[b]
\noindent\hbox to\linewidth{\quad
\begin{tikzpicture}[scale=0.65,x={(7.0cm,-1cm)},y={(2.3cm,2.5cm)},z={(0cm,3cm)}]
\coordinate (A) at (0.649,0.423,0);
\draw[very thin,opacity=0.5] (0,0.928,0.928-1.072)--(0,0.928,0) -- (0,0,0) --
(1,0,0)--(1,0,-0.8);
\draw[very thin,opacity=0.5] (0,0,0)--(0,0,-1.072);
\draw[very thin,opacity=0.5] (0,0.423,0) -- (A) -- (0.649,0,0)
    (0,0.928,0)--(1,0.928,0)--(1,0,0);
\draw (0.649,0.928,0)--(1,0.928,0)--(1,0.423,0);
\draw[very thin,opacity=0.5] (0.649,0,0)--(0.649,0,-0.423)
                              (0,0.423,0)--(0,0.423,-0.649);
\draw (0,0.423,-0.649)--(A)--(0.649,0,-0.423);
\draw (0.649,0.928,0)--(A)--(1,0.423,0);
\draw (0,0,-1.072) -- (1,0,-0.072) (0,0,-1.072) -- (0,0.928,0.928-1.072);

\draw (A) node {$\bullet$};
\draw (A)+(0.06,0.12,0) node {$T$};
\draw[thick] (0,0.928,0.928-1.072)
\foreach \x/\y in {
0.000/0.928,0.059/0.892,0.117/0.856,0.174/0.818, 
0.231/0.778,0.286/0.738,0.341/0.696,0.395/0.654, 
0.448/0.610,0.500/0.565,0.551/0.518,0.600/0.471, 
0.649/0.423,0.697/0.373,0.744/0.323,0.789/0.271, 
0.834/0.219,0.877/0.166,0.919/0.111,0.960/0.056, 
1.000/0.000}
{   -- (\x,\y,\x+\y-1.072) 
};

\draw[dashed] (1,0.423,0)
\foreach \x/\y in {0.649/0.423,0.697/0.373,0.744/0.323,0.789/0.271, 
0.834/0.219,0.877/0.166,0.919/0.111,0.960/0.056, 
1.000/0.000} {
    -- (1,\y,\x+\y-1.072)
};
\draw[dashed] (0,0.928,0.928-1.072)
\foreach \x/\y in {0.000/0.928,0.059/0.892,0.117/0.856,0.174/0.818, 
0.231/0.778,0.286/0.738,0.341/0.696,0.395/0.654, 
0.448/0.610,0.500/0.565,0.551/0.518,0.600/0.471,0.649/0.423} 
{  -- (\x,0.928,\x+\y-1.072)
};
\foreach \x/\y in {0.744/0.323,0.834/0.219,0.919/0.111}
{
  \draw[very thin,opacity=0.3] 
  (0,\y,\y-1.072)--(\x,\y,\x+\y-1.072)--(1,\y,\x+\y-1.072);
}
\foreach \x/\y in {0.117/0.856,0.231/0.778,0.341/0.696,0.448/0.610,0.551/0.518}
{ \draw[very thin,opacity=0.3]
  (\x,0,\x-1.072)--(\x,\y,\x+\y-1.072)--(\x,0.928,\x+\y-1.072);
}
\end{tikzpicture}%
\hss}\kern -6pt%
\caption{Concave case}\label{fig:concave}%
\end{figure}
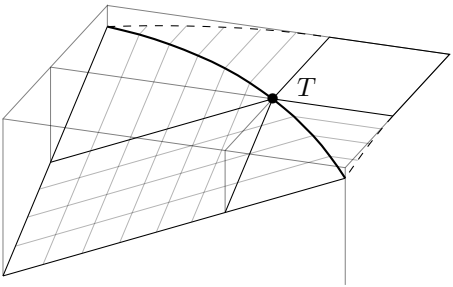
The function $f(x,y)$ is defined as $f(x,y)=t_x+t_y$ if both $x\ge t_x$ and
$y\ge t_y$, otherwise
$$
  f(x,y)=\left\{\mkern-4mu\begin{array}{ll}
    y+\min\{x,\beta(y)\} & \text{if $x\ge t_x$,}\\
    x+\min\{y,\alpha(x)\} & \text{if $y\ge t_y$,}\\
    x+y & \text{otherwise.}
\end{array}\right.
$$
This is an EL function. For example, for a fixed $x\ge t_x$ it is increasing
and concave as $y+\beta(y)$ is increasing on the $[0,t_y]$ interval ($\beta'(y)\le -1$
here), and is concave since $\beta$ is concave. The left and right
partial derivatives of $f$ at a point $(x,y)$ of $S$ with $x\ge t_x$ are $1$
and $0$, and $1$ and $1+\beta'(y)$, respectively. The difference between the
corresponding pairs is at least
$-\beta'(y)\ge -\beta'(0)$. Choosing the multiplier $k$ such that
$k\cdot(-\alpha'(0))\ge 1$ and $k\cdot(-\beta'(0))\ge 1$, the EL function $k\cdot
f$ will be $S$-feasible. The minimal such $k$ gives a cost $k$ EL function
which again matches the lower bound of Theorem \ref{thm:main}.

\section{Conclusion}\label{sec:conclusion}

A continuous version of the discrete Shannon entropy functions, called
\emph{entropy-like}, or EL functions, has been defined in Definition
\ref{def:EL-function}. They form a natural subclass of multivariate
continuous submodular functions which gained considerable attention recently
\cite{bach-submodular}. Interestingly, the same subclass emerged as a
crucial one when investigating possible parallelization of traditional
submodular optimization algorithms \cite{bian2020continuous}.

Motivated by difficult problems in multipartite secret sharing
\cite{padro:idealmulti}, points in the non-negative orthant are flagged as either
qualified or unqualified, separated by a \emph{secret sharing surface $S$}, see
Definition \ref{def:s-surface}. An EL function is \emph{feasible} for such a
surface $S$ if at internal points of $S$ all partial derivatives drop by at
least one  when passing from left to right.
The following optimization problem was considered: for a given
s-surface $S$ find the minimal cost of an $S$-feasible EL function. The first open
problem is to prove that this function set is never empty.
\begin{problem}\label{problem:1}
Prove that for every s-surface $S$ there exists at least one $S$-feasible function.
\end{problem}
\noindent
The cost of an EL function $f$ is the maximum of its partial derivatives at
zero, thus it can be $+\infty$. Definition \ref{def:s-surface} stipulates
that for every $S$-surface there is a positive constant $c$ such that
$1/c<\nablaS_i(x)<c$ at each point $x\in S$. The value in Theorem
\ref{thm:main} bounding the cost of any $S$-feasible function from below 
is smaller than 
$c^2$, thus it does not exclude the following strengthening of Problem
\ref{problem:1}:
\begin{problem}\label{problem:2}
Prove that for every s-surface $S$ there is at least one $S$-feasible
function with \emph{finite cost}.
\end{problem}
The lower bound on the cost of $S$-feasible EL functions proved in
Theorem \ref{thm:main} was shown to be tight for linear s-sur\-faces, and also
for two-dimensional convex and concave s-sur\-faces.
\begin{problem}\label{problem:3}
Find an s-surface $S$ for which the bound in Theorem \ref{thm:main} is not
tight.
\end{problem}
\noindent
As a strenghtening of Problem \ref{problem:3} we offer a bold conjecture which might easily turn out to be false.
\begin{problem}
If $S$ is neither convex nor concave, then the bound of Theorem
\ref{thm:main} is not tight.
\end{problem}
Constructions in Section \ref{sec:bipartite} settled the problem of
finding the optimal values for
two-di\-men\-sion\-al
convex and concave s-surfaces. It would be interesting to
see optimal solutions for convex and concave surfaces in higher
dimensions.
\begin{problem}
Determine the optimal costs of convex and con\-cave s-surfaces in dimension $>2$.
\end{problem}

As mentioned in Section \ref{sec:problem}, the cost function considered in
this paper stems from the \emph{worst case complexity} of general secret
sharing schemes. An alternate cost function corresponding to the
total entropy would be
$$
   \cost^t(f) = \sup \{\, f(x): x \in\dom(f)\,\}.
$$ 
As an EL function can be truncated, the sup here can be limited to the points of
$S$. The two costs functions are obviously related, but it is not clear how
this relationship can be used to connect the corresponding optimization
problems.

\begin{problem}
Prove lower bounds, similar to Theorem \ref{thm:main}, for the optimization
problem $\OPT^t(S)$ using the $\cost^t$ function.
\end{problem}

By Theorem \ref{thm:inf-taken}, if there is any $S$-feasible function at all 
then there is one with minimal cost. The proof relied on the fact that finite
cost EL functions have bounded derivatives. For $\cost^t$ this property does not
hold anymore.
\begin{problem}
If there is an
$S$-feasible function, then there is one  with $\cost^t(f)=\OPT^t(S)$.
\end{problem}

Finally, extend the quite meager collection of s-surfaces from Section
\ref{sec:bipartite} for which the exact bound is known.
\begin{problem}
Find optimal solutions for additional ``interesting'' s-surfaces for both
cost functions.
\end{problem}

\section*{Acknowledgment}
The work of the author was supported by the GACR grant number 19-045798, 
and by the Lend\"ulet program of the HAS.

\bibliographystyle{plain}
\bibliography{22}

\end{document}